\documentclass[12pt]{amsart}
\usepackage{amsmath}
\usepackage{amssymb,amsthm}
\usepackage{xspace,color,dsfont}

\theoremstyle{plain}      
\newtheorem{thm}{Theorem}
\newtheorem{lem}[thm]{Lemma}

\makeatletter
\renewcommand\subsection{\@startsection{subsection}{2}%
  \z@{.5\linespacing\@plus.7\linespacing}{-.5em}%
  {\normalfont\scshape}}
\makeatother

\begin{document}
\title[Patterns of conjunctive forks]{\large Patterns of conjunctive forks\\[0.5ex]~~}
\author{Va\v{s}ek Chv\'{a}tal, Franti\v sek Mat\' u\v s and Yori Zw\'ol\v{s}}
\address{Department of Computer Science and Software Engineering\\
         Concordia University, Montreal, Quebec H3G 1M8\\ Canada}
         \email[V.\ Chv\'{a}tal]{chvatal@cse.concordia.ca}
\thanks{Research of V.\ Chv\'{a}tal and Y.\ Zw\'ol\v{s} was supported by the Canada Research
        Chairs program and by the Natural Sciences and
        Engineering Research Council of Canada}
\address{Institute of Information Theory and Automation\\
         Academy of Sciences of the Czech Republic\\
         Pod vod\' arenskou v\v e\v z\'{\i} 4, 182 08 Prague~8\\
         Czech Republic}
         \email[F.\  Mat\' u\v s]{matus@utia.cas.cz}
\thanks{Research of F.\  Mat\' u\v s was supported by
          Grant Agency of the Czech Republic under Grant 13-20012S}
\address{Google DeepMind, London, United Kingdom}
         \email[Y.\ Zw\'ol\v{s}]{yori@google.com}
\begin{abstract}
    Three events in a probability space form a conjunctive fork if
    they satisfy specific constraints on conditional independence
    and covariances. Patterns of conjunctive forks within collections
    of events are characterized by means of
    systems of linear equations that have positive solutions. This characterization allows 
		patterns of conjunctive forks to be recognized in polynomial time.
		Relations
    to previous work on causal betweenness and
    on patterns of conditional independence among random variables
    are discussed.
\end{abstract}

\maketitle

\newcommand{\pmn}{\emptyset}
\newcommand{\pdm}{\subseteq}
\newcommand{\sm}{\setminus}
    \renewcommand{\ge}{\geqslant}    
		\renewcommand{\le}{\leqslant}
    \renewcommand{\geq}{\geqslant}   
		\renewcommand{\leq}{\leqslant}
    \def\Z{\mathbb{Z}}
\renewcommand{\Pr}{\ensuremath{P}\xspace}
\newcommand{\Qr}{\ensuremath{Q}\xspace}
\newcommand{\Exp}{\ensuremath{\mathrm{E}}\xspace}
    \newcommand{\egy}{\ensuremath{\mathds{1}}}
    \newcommand{\ci}{\ensuremath{\!\perp\!\!\!\!\perp\!}}
\newcommand{\qqq}{\ensuremath{{\mathfrak q}}\xspace}
\newcommand{\rrr}{\ensuremath{\mathfrak r}\xspace}
\newcommand{\bbb}{\ensuremath{\mathfrak b}\xspace}
\newcommand{\sss}{\ensuremath{\mathfrak s}\xspace}
\newcommand{\OO}{\ensuremath{\mathit\Omega}\xspace}
\newcommand{\zlo}[2]{\ensuremath{\mbox{\large$\frac{#1}{#2}$}}}
\newcommand{\abs}[1]{\lvert#1\rvert}

\newcommand{\peq}{\mathrel{\smash{\overset{\scriptscriptstyle\Pr}{=}}}}
\newcommand{\eee}{\ensuremath{\mathrel{\smash{\overset{\scriptscriptstyle\rrr}{\thicksim}}}}\xspace}
\newcommand{\eeq}{\ensuremath{\mathrel{\smash{\overset{\scriptscriptstyle\qqq}{\thicksim}}}}\xspace}
\newcommand{\oA}{\overline{A}}
\newcommand{\oB}{\overline{B}}
\newcommand{\oC}{\overline{C}}
\newcommand{\oD}{\overline{D}}
\newcommand{\oE}{\overline{E}}
\newcommand{\oF}{\overline{F}}
\newcommand{\vare}{\ensuremath{\varepsilon}\xspace}
\newcommand{\corr}{\ensuremath{\mathrm{corr}}\xspace}
\newcommand{\cvr}[2]{\ensuremath{\mathrm{cov}({#1},{#2})}\xspace}
\section{Motivation}

 Hans Reichenbach \cite[Chapter 19]{Rei56} defined a \emph{conjunctive fork}
 as an ordered triple $(A,B,C)$ of events $A$, $B$ and $C$ in a probability
 space $(\OO,\mathcal F,\Pr)$ that~satisfies
 \begin{align}
    \Pr(AC|B) &= \Pr(A|B) \Pr(C|B)\,,\label{reich1}\\
    \Pr(AC|\oB) &= \Pr(A|\oB) \Pr(C|\oB)\,,\label{reich2}\\
    \Pr(A|B) &> \Pr(A|\oB)\,,\label{reich3}\\
    \Pr(C|B) &>  \Pr(C|\oB)\,\label{reich4}
 \end{align}
 where, as usual, $AC$ is a shorthand for $A\cap C$ and $\oB$ denotes the
 complementary event $\OO\sm B$. (Readers comparing this definition with Reichenbach's
    original beware: his notation is modified here by switching the role of
    $B$ and $C$. To denote the middle event in the fork, he used $C$, perhaps
    as mnemonic for `common cause'.)
Implicit in this definition is the asumption
\begin{equation}
0<\Pr(B)<1\, ,\label{reich5}
\end{equation}
which is needed to define the conditional probabilities in~\eqref{reich1}--\eqref{reich4}.

A similar notion
 was introduced earlier in the context of sociology \cite[Part~I,
 Section 2]{KenLaz50}, but the context of Reichenbach's discourse was philosophy of science: conjunctive
 forks play a central role in his causal theory of time. In this
 role, they have attracted considerable attention: over one hundred
 publications, such as \cite{Bre77,Sal80,Sal84,EllEri86,CarJon91,
 Dow92,Spo94,HofRedSza99,Kor99}, refer to them.   Yet for all this interest, no one seems to have asked  
 a fundamental question: 
\begin{center}
	\fbox{\emph{What do ternary relations defined by conjunctive forks look like?}}
\end{center}
The purpose of our paper is to answer this question.

\medskip

An additional stimulus to our work was a previous answer~\cite{ChvWu12} to a similar question,
\begin{center}
\fbox{\emph{What do ternary relations defined by causal betweenness look like?}}
\end{center}
Here, \emph{causal betweenness} is another ternary relation on sets of events in probability spaces, also introduced by Reichenbach\cite[p.~190]{Rei56} in the context of his causal theory of time. 
(This relation is reviewed in Section~\ref{S:btw}.) From this perspective, our paper may be seen as a companion to~\cite{ChvWu12}.

\section{The main result}

Let us write $(A,B,C)_{\Pr}$ to signify that
 $(A,B,C)$ is a conjunctive fork in a probability
 space $(\OO,\mathcal F,\Pr)$ and let us say that
 events  $A_i$ in this space, indexed by elements $i$ of a set $N$, 
\emph{fork-represent} a ternary relation $\rrr$ on $N$ if and only if
 \[
    \rrr=\{(i,j,k)\in N^3\colon (A_i,A_j,A_k)_{\Pr}\,\}\,.
 \]
In order to characterize ternary relations on a finite ground set that are fork representable, we need a few definitions.

To begin, call a ternary relation \rrr
 on a ground set~$N$ a \emph{forkness} if and only if it satisfies 
\begin{align}
    (i,j,i)\in\rrr\;\;&\Rightarrow\;\;(j,i,j)\in\rrr\label{sym}\\
    (i,j,i), (j,k,j)\in\rrr
                        \;\;&\Rightarrow\;\;(i,k,i)\in\rrr\label{E:trans}\\
    (i,k,j)\in\rrr\;\;&\Rightarrow\;\;(j,k,i)\in\rrr\label{flip}\\
		(i,j,k)\in\rrr\;\;&\Rightarrow\;\;(i,j,j), (j,k,k),(k,i,i)\in \rrr \label{lower}\\    
		 (i,k,j), (i,j,k)\in\rrr\;\;&\Rightarrow\;\; (j,k,j)\in\rrr \label{btw}
 \end{align}
 for all choices $i,j,k$ in $N$. 

Given a forkness $\rrr$, write
\[
V_\rrr =\{i\in N\colon(i,i,i)\in\rrr\}
\]
and let \eee denote the binary relation defined on $V_\rrr$ by 
 \[
    i\eee j\;\;\Leftrightarrow\;\;(i,j,i)\in\rrr\,.
 \] 
 This binary relation is reflexive by definition of $V_\rrr$, it
 is symmetric by~\eqref{sym}, and it is transitive by~\eqref{E:trans}. In short, \eee is
 an equivalence relation. Call a forkness \rrr \emph{regular} if, and only if,
 \begin{equation}\label{E:regular}
    (i,j,k)\in\rrr, \; i\eee i', \; j\eee j', \; k\eee k'
            \;\Rightarrow\; (i',j',k')\in \rrr\,.
 \end{equation}
The \emph{quotient} of a a regular forkness \rrr is the the ternary relation whose ground set is the set of equivalence classes of \eee 
and which 
consists of all triples $(I,J,K)$ such that $(i,j,k)\in\rrr$
 for at least one $i$ in $I$, at least one $j$ in $J$, and at least one $k$ in $K$. (Equivalently, since \rrr is regular, $(I,J,K)$ belongs to its quotient if and only if 
$(i,j,k)\in\rrr$ for all $i$ in $I$, for all $j$ in $J$,~and~for all~$k$~in~$K$.)  

 Call a ternary relation \qqq \emph{solvable} if and only if the linear system 
 \begin{multline}\label{system}
     x_{\{I,K\}}\!=\!x_{\{I,J\}}+x_{\{J,K\}}\\ 
        \text{ for all $(I,J,K)$ in $\qqq$ with pairwise distinct $I,J,K$}
 \end{multline}
 has a solution with each $x_{\{I,J\}}$ positive.

\begin{thm}\label{T:ChMaZw}
    A ternary relation on a finite ground set is fork representable if and only
    if it is a regular forkness and its quotient is solvable.
\end{thm}
Theorem~\ref{T:ChMaZw} implies that fork-representability of a ternary
 relation~$\rrr$ on a finite ground set $N$ can be tested in time polynomial
 in~$\abs{N}$. More precisely, polynomial time suffices to test 
 \rrr for being a forkness, for testing its regularity, and for the construction of its quotient \qqq. Solvability of \qqq means solvability of a system
 of linear equations and linear inequalities, which can be tested in polynomial
 time by the breakthrough result of~\cite{Kha79}.

\medskip

We prove the easier `only if' part of Theorem~\ref{T:ChMaZw} in in Section~\ref{S:onlyif} and we prove the `if' part in Section~\ref{S:mainproof}.
In Section~\ref{S:btw}, we comment on causal betweenness and its relationship to causal forks.  
In the final Section~\ref{S:disc},  we discuss connections to previous work on 
 patterns of conditional independence.

\section{Proof of the `only if' part}\label{S:onlyif}

\subsection{Reichenbach's definition restated}  
Reichenbach's definition of a conjunctive fork has a neat paraphrase in terms of random variables. To present it, let us first review a few standard definitions. 

The \emph{indicator function\/} $\egy_E$ of an event $E$ in a probability space is the random variable defined by $\egy_E(\omega)=1$ if $\omega\in E$
 and $\egy_E(\omega)=0$ if $\omega\in \oE$. Indicator functions $\egy_A$ and $\egy_C$ are said to be \emph{conditionally independent given $\egy_B$,\/} in symbols $\egy_A\ci\egy_C|\egy_B$, if and only if events $A$, $B$, $C$ satisfy~\eqref{reich1} and~\eqref{reich2}. 
The \emph{covariance\/} of $\egy_A$ and $\egy_B$, denoted here as $\cvr{A}{B}$, is defined by 
 $$\cvr{A}{B}=\Pr(AB)-\Pr(A)\Pr(B).$$ 
Since~\eqref{reich3} means that  $\cvr{A}{B}>0$ and~\eqref{reich4} means that $\cvr{B}{C}>0$, we conclude that
\begin{multline}
(A,B,C)_\Pr \Leftrightarrow 
 \egy_A\ci\egy_C|\egy_B \;\&\; \cvr{A}{B}\!\!>\!\!0 \;\&\; \cvr{B}{C}\!\!>\!\!0. \label{fork}
\end{multline}

\subsection{A couple of Reichenbach's results}  

\ Reichenbach~\cite[p.~160, equation~(12)]{Rei56}  noted that
\begin{multline}\label{rr1}
\egy_A\ci\egy_C|\egy_B \;\Rightarrow\;\\
\cvr{A}{C}=\cvr{B}{B}\cdot(\Pr(A|B)-\Pr(A|\oB))\cdot(\Pr(C|B)-\Pr(C|\oB))
\end{multline}
and so~\cite[p.~158, inequality (1)]{Rei56}
\begin{equation}\label{covac}
(A,B,C)_\Pr \Rightarrow \cvr{A}{C}>0. 
\end{equation}
 Implication~\eqref{covac} was his reason for calling the fork
    `conjunctive': it is ``a fork which makes the conjunction of the two
    events more frequent than it would be for independent events''
    \cite[p.~159]{Rei56}.

Since 
\begin{multline*}
\cvr{B}{B}(\Pr(A|B)-\Pr(A|\oB))=\\ \Pr(AB)(1-\Pr(B))-\Pr(A\oB)\Pr(B)=\cvr{A}{B}
\end{multline*}
and, similarly, $\cvr{B}{B}(\Pr(C|B)-\Pr(C|\oB))=\cvr{C}{B}$, 
Reichenbach's implication~\eqref{rr1} can be stated as
\begin{equation}\label{fero}
\egy_A\ci\egy_C|\egy_B \;\Rightarrow\;\cvr{A}{C}\cdot\cvr{B}{B}=\cvr{A}{B}\cdot\cvr{B}{C}.
\end{equation}
\subsection{The idea of the proof}  
An event $E$ is called \emph{\Pr-nontrivial\/} if and only if $0<\Pr(E)<1$, which is equivalent to $\cvr{E}{E}>0$. 
When $E,F$ are \Pr-nontrivial events, the \emph{correlation\/} of their indicator functions, denoted here as $\corr(E,F)$, is defined 
by
\[
\corr(E,F)\;=\; \frac {\cvr{E}{F}}{\cvr{E}{E}^{1/2}\cvr{F}{F}^{1/2}}\, .
\]
In these terms, \eqref{fero} reads
\begin{equation}\label{corident}
\egy_A\ci\egy_C|\egy_B \;\Rightarrow\;\corr(A,C)=\corr(A,B)\cdot \corr(B,C).
\end{equation}

The strict inequalities \eqref{reich3}, \eqref{reich4}, \eqref{reich5} imply that
\begin{multline}\label{nontriv}
\text{in every conjunctive fork $(A,B,C)$,}\\ \text{all three events $A$, $B$, $C$ are \Pr-nontrivial,}
\end{multline}
and so $\corr(A,B)$, $\corr(A,C)$, $\corr(B,C)$ are well defined;
\eqref{fork} guarantees that $\corr(A,B)>0$, $\corr(B, C)>0$ and \eqref{covac} guarantees that $\corr(A,C)>0$. 

Fact~\eqref{corident} guarantees that the system
\[
x_{\{A,C\}}\!=\!x_{\{A,B\}}+x_{\{B,C\}} 
        \text{ for all conjunctive forks $(A,B,C)$}
\]
can be solved by setting $x_{\{E,F\}} = -\ln \corr(E,F)$.
This observation goes a long way toward proving the `only if' part of Theorem~\ref{T:ChMaZw}, but it does not quite get there: 
For instance, if $(A,B,C)$ is a conjunctive fork and $A\peq B\,$, then $\corr(A,B)=1$, and so $x_{\{A,B\}}=0$, but	the `only if' 
part of the theorem requires $x_{\{A,B\}}>0$. To get around such obstacles, we deal with 
the quotient of the ternary relation made from conjunctive forks.

\subsection{Other preliminaries}		
 Events $E$ and $F$ are said to be \Pr-equal, in symbols $E\peq F$, if and only
 if $\Pr(E{\vartriangle}F)=0$. We claim that
\begin{align}
& (A,B,A)_{\Pr} \text{ if and only if $A$ is \Pr-nontrivial and $A\peq B\,$,}\label{iji} \\ 
& \text{if $(A,B,C)_{\Pr}$ and $(A,C,B)_{\Pr}$, then $B\peq C$.}\label{fbtw}
\end{align} 
To justify claim~\eqref{iji}, note that $(A,B,A)_{\Pr}$ means the conjunction of $\cvr{A}{B}>0$  and at least one of 
$\Pr(A)=0$, $\Pr(A)=1$, $A\peq \oB$, $A\peq B$; of the four equalities, only the last one is compatible with $\cvr{A}{B}>0$.
To justify claim~\eqref{fbtw}, note that $\cvr{B}{B}-\cvr{B}{C}=\Pr (B)\Pr (\oB C) +\Pr (\oB)\Pr (B\oC)$; when $B$ is 
$\Pr$-nontrivial, the right-hand side vanishes if and only if $\Pr(B{\vartriangle}C)=0$; it follows that
\begin{align*}
& \text{if $(A,B,C)_{\Pr}$, then $\cvr{B}{B}\geq\cvr{B}{C}$}\\
& \text{with equality if and only if $B\peq C\,$;}
\end{align*}
by \eqref{fero}, this implies that  
\begin{align*}
& \text{if $(A,B,C)_{\Pr}$, then $\cvr{A}{B}\geq\cvr{A}{C}$}\\
& \text{with equality if and only if $B\peq C\,$,}
\end{align*}
which in turn implies~\eqref{fbtw}.

\subsection{The proof}

\begin{lem}\label{V1}
A ternary relation is fork representable only
    if it is a regular forkness.
		\end{lem}
\begin{proof}
Consider a fork-representable ternary relation \rrr on a ground set $N$. Proving the lemma means verifying that \rrr 
has properties \eqref{sym}, \eqref{E:trans}, \eqref{flip}, \eqref{lower}, \eqref{btw}, \eqref{E:regular}. 
Since \rrr is fork representable, there are events  $A_i$ in in some probability
 space $(\OO,\mathcal F,\Pr)$, with $i$ ranging over $N$, such that
\[
(i,j,k)\in \rrr \;\Leftrightarrow\;
 (A_i,A_j,A_k)_{\Pr}\, .
\]
Properties \eqref{sym} and  \eqref{E:trans}, 
\begin{align*}
& (i,j,i)\in\rrr\;\;\Rightarrow\;\;(j,i,j)\in\rrr\, ,\\
& (i,j,i), (j,k,j)\in\rrr\;\;\Rightarrow\;\;(i,k,i)\in\rrr\, ,	
\end{align*}
follow from \eqref{iji}. Property~\eqref{flip},
\begin{align*}
& (i,k,j)\in\rrr\;\;\Rightarrow\;\;(j,k,i)\in\rrr\, ,
\end{align*}
is implicit in the definition of $(A_i,A_j,A_k)_{\Pr}$. Property~\eqref{lower},
\begin{align*}
& (i,j,k)\in\rrr\;\;\Rightarrow\;\;(i,j,j), (j,k,k),(k,i,i)\in \rrr\, , 
\end{align*}
follows from~\eqref{covac}. Property \eqref{btw},
\begin{align*}
& (i,j,k)\in\rrr\;\;\Rightarrow\;\;(i,j,j), (j,k,k),(k,i,i)\in \rrr\, , 
\end{align*}
follows from \eqref{fbtw} and \eqref{iji}. Property~\eqref{E:regular},
\[
    (i,j,k)\in\rrr, \; i\eee i', \; j\eee j', \; k\eee k'
            \;\Rightarrow\; (i',j',k')\in \rrr\, ,
\]
follows from~\eqref{iji} alone. 
\end{proof}

\begin{lem}\label{V2}
A regular forkness is fork representable only if its quotient is solvable.
		 \end{lem}
 \begin{proof}
Consider a fork-representable regular forkness \rrr on a ground set $N$. 
Since \rrr is fork representable, there are events  $A_i$ in in some probability
 space $(\OO,\mathcal F,\Pr)$, with $i$ ranging over $N$, such that
\[
(i,j,k)\in \rrr \;\Leftrightarrow\;
 (A_i,A_j,A_k)_{\Pr}\, .
\]
With \qqq standing for the quotient of \rrr, proving the lemma means finding 
positive numbers $x_{\{I,J\}}$ such that 
\begin{multline}\label{solution}
  x_{\{I,K\}}\!=\!x_{\{I,J\}}+x_{\{J,K\}}\\ 
        \text{ for all $(I,J,K)$ in $\qqq$ with pairwise distinct $I,J,K$.}
\end{multline}
We claim that this can be done by first choosing an element $r(I)$ from each equivalence class $I$ of \eee and then setting $$x_{\{I,J\}} = -\ln \corr({A_{r(I)}},{A_{r(J)}})$$
for every pair $I,J$ of distinct equivalence classes that appear together in a triple in \qqq. 
(By \eqref{iji}, the right hand side depends only on $I$ and $J$ rather than the choice of $r(I)$ and $r(J)$.)

To justify this claim, note first that~\eqref{solution} is satisfied by virtue of~\eqref{corident}.
A special case of the {\em covariance inequality\/} guarantees that every pair $A,B$ of events satisfies 
\[
(\cvr{A}{B})^2 \le \cvr{A}{A}\cdot \cvr{B}{B}
\]
and that the two sides are equal if and only if 
    $A\peq B$ or $A\peq \oB$ or $\Pr(A)=0$ or $\Pr(A)=1$ or $\Pr(B)=0$ or $\Pr(B)=1$. 
		If equivalence classes $I,J$ of \eee are distinct, then $A_{r(I)}\not\peq A_{r(J)}$; if, in addition, $I$ and $J$ appear together 
		in a triple in \qqq, then $A_{r(I)}\not\peq \oA_{r(J)}$ (since $\corr({A_{r(I)}},{A_{r(J)}})>0$ by~\eqref{fork} and~\eqref{covac}) and 		
		$0<\Pr(A_{r(I)})<1$, $0<\Pr(A_{r(J)})<1$ (by~\eqref{nontriv}). In this case, $$\cvr{A_{r(I)}}{A_{r(J)}}^2<\cvr{A_{r(I)}}{A_{r(I)}}\cdot\cvr{A_{r(J)}}{A_{r(J)}},$$
		and so $x_{\{I,J\}}>0$.
\end{proof}

\section{Proof of the `if' part}\label{S:mainproof}

\begin{lem}\label{V3}
If the quotient of a regular forkness on a finite ground set is solvable, then it is fork representable.
		 \end{lem}
 
\begin{proof}
Given the quotient \qqq of a regular forkness on a finite ground set along with 
positive numbers $x_{\{I,J\}}$ such that 
\begin{multline}
  x_{\{I,K\}}\!=\!x_{\{I,J\}}+x_{\{J,K\}}\\ 
        \text{ for all $(I,J,K)$ in $\qqq$ with pairwise distinct $I,J,K$,}\label{solution++}
\end{multline}
we have to construct a probability
 space $(\OO,\mathcal F,\Pr)$ and 
 events  $A_I$ in this space, indexed by elements $I$ of the ground set $C$ of \qqq, such that 
 \begin{equation}\label{qeq}
    \qqq=\{(I,J,K)\in C^3\colon (A_I,A_J,A_K)_{\Pr}\,\}\,.
 \end{equation}

For this purpose, we let $\OO$ be the power set $2^C$ of $C$, we let $\mathcal F$ be the power set $2^{\OO}$ of $\OO$, and we set 
 \[ 
A_I=\{\omega\in\OO: I\in\omega\}.
\]
Now let us construct the probability measure $\Pr$. Set $n=\abs{C}$. 
Given a subset $L$ of $C$, consider the function $\chi_L\colon \OO\rightarrow \{+1, -1\}$ defined by 
  \[
    \chi_L(\omega)= (-1)^{\abs{\,\omega\,\cap\, L\,}}\,.
   \]
Let $E_\qqq$ stand for the family of all two-element subsets
 $\{I,J\}$ of $C$ such that $I$ and $J$ appear together in a triple in \qqq.
Let $M_\qqq$ stand for the family of all three-element subsets
 $\{I,J,K\}$ of $C$ such that $\{I,J\}$, $\{J,K\}$, and $\{K,I\}$ belong to $E_\qqq$ 
and no triple in $\qqq$ is formed by all three $I,J,K$.
 Finally, for positive numbers $\gamma$ and $\vare$ that are sufficiently small in a sense to be specified shortly, 
define $\Pr\colon \OO\rightarrow {\bf R}$ by 
 \[
    \Pr(\omega)=2^{-n}\Big[1+
    \textstyle{\sum_{\{I,J\}\in E_\qqq}\;\chi_{\{I,J\}}(\omega)\,\gamma^{x\{I,J\}}
    +\vare\sum_{\{I,J,K\}\in M_\qqq}}\;\chi_{\{I,J,K\}}(\omega)\Big].
 \] 
(In exponents, we write $x\{I,J\}$ in place of $x_{\{I,J\}}$.)
     Here, readers with background in harmonic analysis will have recognized
    the Fourier-Stieltjes transform on the group $\Z_2$ and the characters
    $\chi_L$.
		
 When $L$ is nonempty, $\chi_L$ takes each of the values $\pm1$ on the same
 number its arguments $\omega$, and so  $\sum_{\omega\in{\mathit\Omega}}\chi_L(\omega)=0$, which implies that 
 $\sum_{\omega\in{\mathit\Omega}}P(\omega)=1$. If $\gamma<n^{-2/x\{I,J\}}$
 for all $\{I,J\}\in E_\qqq$ and $\vare<n^{-3}$, then
 \[
    2^{n}\Pr(\omega)\geq
        1-\abs{E_\qqq}\,n^{-2}-\abs{M_\qqq}\,n^{-3}>0\,,
 \]
so that \Pr is a probability measure, positive on the elementary events.

  The bulk of the proof consists of verifying that this construction satisfies~\eqref{qeq}. 
	We may assume that $C\ne\emptyset$ (else $\OO$ is a singleton, and so \eqref{qeq} holds).
	Given a subset $S$ of $C$, write
	\[
	A_S \;=\; \{\omega\in\OO: S\subseteq\omega\}.
	\]
	Since
 \[
    \textstyle{\sum_{\omega\in A_S}\chi_L(\omega)}=
        \begin{cases}
            \;(-1)^{|L|}\,2^{n-\abs{S}}& \quad\text{if $L\subseteq S$},\\
            \;0                        & \quad\text{otherwise},
        \end{cases}
 \]
 we have
\begin{multline}\label{master}
    2^{\abs{S}}\Pr(A_S)= \\
		1+ \textstyle{\sum_{\{I,J\}\in E_\qqq,\:I,J\in S}\,\gamma^{x\{I,J\}}}
            -\vare\,\big|\{\{I,J,K\}\in M_\qqq\colon I,J,K\in S\}\big|\,.
 \end{multline}
In particular, formula \eqref{master} yields for all $I$ in $C$
\[
\Pr(A_I)=\frac12 
\]  
and it yields for all choices of distinct $I,J$ in $C$ 
\begin{equation}\label{f2}
    \Pr(A_{\{I,J\}})= \frac14+
                \begin{cases}
                    \;\frac14\gamma^{\,x\{I,J\}} & \quad\text{if $\{I,J\}\in E_\qqq$},\\
                    \;0\,  &                     \quad\text{otherwise.}
                \end{cases}
\end{equation}
It follows that
\begin{equation}\label{cvrgamma}
    \cvr{A_I}{A_J}=
        \begin{cases}
                \;\frac14\;\gamma^{\,x(\{I,J\})} & \text{if } \{I,J\}\in E_\qqq,\\
                \;0     & \text{otherwise.}
        \end{cases}
\end{equation}
For future reference, note also that, by definition,
\begin{align}\label{qfork}
&\text{\qqq is a forkness such that $(I,I,I)\in\qqq$ for all $I$ in $C$}\\
&\text{and such that $(I,J,I)\not\in\qqq$ whenever $I\ne J$.}\nonumber
\end{align}
and that 
\begin{align}\label{ijj}
& \{I,J\}\in E_\qqq \;\Leftrightarrow\; (I,J,J) \in\qqq.
\end{align}
(here, implication $\Rightarrow$ follows from properties~\eqref{flip} and \eqref{lower} of forkness 
and implication $\Leftarrow$ follows straight from the definition of $E_\qqq$).

Now we are ready to verify~\eqref{qeq}. Given a triple $(I,J,K)$ in $C^3$, we have to show that
\begin{equation}\label{verq}
  		(A_I,A_J,A_K)_{\Pr} \Leftrightarrow (I,J,K)\in\qqq\, . 
\end{equation}

{\sc Case 1:} $I=J=K$.\\
Since $C\ne\emptyset$, all $I$ in $C$ are $\Pr$-nontrivial, and so 
we have $(A_I,A_I,A_I)_{\Pr}$. By~\eqref{qfork}, we have $(I,I,I)\in\qqq$. 
 
{\sc Case 2:} $I\ne J$, $K=I$.\\
 Here, $A_I\not\peq A_J$, and so~\eqref{iji} implies that
 $(A_I,A_J,A_I)$ is not a conjunctive fork. By \eqref{qfork}, we have $(I,J,I)\not\in\qqq$.

{\sc Case 3:} $I\ne J$, $K=J$.\\
If $\{I,J\}\in E_\qqq$, then~\eqref{cvrgamma} guarantees $\cvr{A_I}{A_J}>0$, which implies $(A_I,A_J,A_J)_{\Pr}$. By~\eqref{ijj}, we have $(I,J,J)\in\qqq$.
 
 If $\{I,J\}\not\in E_\qqq$, then~~\eqref{cvrgamma} guarantees $\cvr{A_I}{A_J}=0$, and so $(A_I,A_J,A_J)$ is not a conjunctive fork. By~\eqref{ijj}, we have $(I,J,J)\not\in\qqq$.
 
{\sc Case 4:} $I= J$, $K\ne J$.\\
This case is reduced to {\sc Case 3} by the flip $I\leftrightarrow K$, which preserves both sides of~\eqref{verq}.

{\sc Case 5:} $I,J,K$ {\em are pairwise distinct and at least one of
 $\{I,J\}$, $\{J,K\}$, $\{K,I\}$ does not belong to $E_\qqq$.\/}\\ 
By~\eqref{cvrgamma}, at least one of the covariances $\cvr{A_I}{A_J}$, $\cvr{A_J}{A_K}$, $\cvr{A_K}{A_I}$ vanishes, 
and so~\eqref{fork}, \eqref{covac} guarantee that $(A_I,A_J,A_K)$ is not a conjunctive fork. By definition of 
$E_\qqq$, we have $(I,J,K)\not\in\qqq$. 

{\sc Case 6:} $I,J,K$ {\em are pairwise distinct and all of $\{I,J\}$, $\{J,K\}$, $\{K,I\}$ belong to $E_\qqq$.\/}\\ 
By~\eqref{cvrgamma}, all of $\cvr{A_I}{A_J}$, $\cvr{A_J}{A_K}$, $\cvr{A_K}{A_I}$ are positive.
Now~\eqref{fork} implies that $(A_I,A_J,A_K)_{\Pr}$ is equivalent to
 $\egy_{A_I}\ci\egy_{A_K}|\egy_{A_J}$, which means the conjunction of
 \begin{align*}
   \tfrac12\Pr(A_{\{I,J,K\}})&=\Pr(A_{\{I,J\}})\Pr(A_{\{J,K\}})\,,\\
   \tfrac12\big[\Pr(A_{\{I,K\}})-\Pr(A_{\{I,J,K\}})\big]&=
    \big[\tfrac12-\Pr(A_{\{I,J\}})\big]\cdot
		\big[\tfrac12-\Pr(A_{\{J,K\}})\big]\,.
 \end{align*}
Substitution from \eqref{f2} converts these two equalities to 
 \begin{align}
    \label{E:vind}
    8\Pr(A_{\{I,J,K\}})=&
        1\!+\!\gamma^{\,x\{I,J\}}\!+\!\gamma^{\,x\{J,K\}}
                \!+\!\gamma^{\,x\{I,J\}+x\{J,K\}}\\
    8\Pr(A_{\{I,J,K\}})=&
        1\!+\!\gamma^{\,x\{I,J\}}\!+\!\gamma^{\,x\{J,K\}}
                \!-\!\gamma^{\,x\{I,J\}+x\{J,K\}}\!+\!2\gamma^{\,x\{I,K\}}.
   \label{E:vind2}
 \end{align}
Conjunction of~\eqref{E:vind} and ~\eqref{E:vind2} is equivalent to the conjunction of~\eqref{E:vind} and 
\begin{equation}\label{E:vind3}
x_{\{I,J\}}+x_{\{J,K\}}=x_{\{I,K\}}\, . 
\end{equation}
To summarize, $(A_I,A_J,A_K)_{\Pr}$ is equivalent to the conjunction of~\eqref{E:vind} and ~\eqref{E:vind3}.

{\sc Subcase 6.1:} $\{I,J,K\}\in M_\qqq$.\\  
In this subcase, formula \eqref{master} yields
\[
    8\,\Pr(A_{\{I,J,K\}})=1+ \gamma^{\,x\{I,J\}}+ \gamma^{\,x\{J,K\}}+ \gamma^{\,x\{I,K\}}-\vare\, ,
 \]
which reduces \eqref{E:vind} to
 $\gamma^{\,x\{I,K\}}-\vare =\gamma^{\,x\{I,J\}+x\{J,K\}}$.
 This is is inconsistent with~\eqref{E:vind3}, and so $(A_I,A_J,A_K)$ is not a conjunctive fork.
 By definition of $M_\qqq$, we have $(I,J,K)\not\in~\qqq$.

{\sc Subcase 6.2:} $\{I,J,K\}\not\in M_\qqq$.\\  
In this subcase, formula \eqref{master} yields
\[
    8\,\Pr(A_{\{I,J,K\}})=1+ \gamma^{\,x\{I,J\}}+ \gamma^{\,x\{J,K\}}+ \gamma^{\,x\{I,K\}} ,
 \]
which reduces \eqref{E:vind} to~\eqref{E:vind3}, and so  
$(A_I,A_J,A_K)_{\Pr}$ is equivalent to~\eqref{E:vind3} alone.
Now completing the proof means verifying that 
\[
  		x_{\{I,J\}}+x_{\{J,K\}}=x_{\{I,K\}} \;\;\Leftrightarrow\;\; (I,J,K)\in\qqq . 
\]
Implication $\Leftarrow$ is ~\eqref{solution++}.
To prove the reverse implication, note first that by~\eqref{solution++} along with $x_{\{I,J\}}>0$ and $x_{\{J,K\}}>0$, we have 
\begin{align}
x_{\{I,J\}}+x_{\{J,K\}}=x_{\{I,K\}}  \;\;\Rightarrow\;\; x_{\{J,K\}}\!<\!x_{\{I,K\}} \;\;\Rightarrow\;\; (J,I,K)\not\in \qqq\, ,\label{jik}\\
x_{\{I,J\}}+x_{\{J,K\}}=x_{\{I,K\}}  \;\;\Rightarrow\;\; x_{\{I,J\}}\!<\!x_{\{I,K\}} \;\;\Rightarrow\;\; (I,K,J)\not\in \qqq\, .\label{ikj}
\end{align}
By assumptions of this case and subcase, some triple in $\qqq$ is formed by all three $I,J,K$ 
 and so, since \qqq is a forkness, ~\eqref{flip}  with \qqq in place of \rrr guarantees that at least one of $(J,I,K)$, $(I,K,J)$, $(I,J,K)$ 
belongs to~\qqq. If $x_{\{I,J\}}+x_{\{J,K\}}=x_{\{I,K\}}$, then \eqref{jik} and \eqref{ikj} exclude the first two options, and so we have 
$(I,J,K)\in\qqq$.
\end{proof}

\begin{lem}\label{V4}
If a regular forkness has a fork-representable quotient, then it is fork representable.
		 \end{lem}
\begin{proof} 
Given a regular forkness \rrr, a probability space $(\OO^0,\mathcal F^0,\Pr^0)$, and 
 events  $A_I^0$ in this space, indexed by elements $I$ of the ground set $C$ of the quotient \qqq of \rrr, such that 
\[
    \qqq=\{(I,J,K)\in C^3\colon (A_I^0,A_J^0,A_K^0)_{\Pr^0}\,\},
 \]		
we have to construct a probability
 space $(\OO,\mathcal F,\Pr)$ and 
 events  $A_i$ in this space, indexed by elements $i$ of the ground set $N$ of \rrr, such that 
 \begin{equation}\label{req}
    \rrr=\{(i,j,k)\in N^3\colon (A_i,A_j,A_k)_{\Pr}\,\}\,.
 \end{equation}		
For this purpose, we let $\OO$ be the power set $2^N$ of $N$, we let $\mathcal F$ be the power set $2^\OO$ of $\OO$,
and we set 
\[ 
A_i=\{\omega\in\OO: i\in\omega\}.
\]
For each element $\omega$ of $\OO$ such that every equivalence class $I$ of \eee satisfies $I\subseteq\omega$ or $I\cap\omega=\emptyset$, define $\Pr(\omega)=\Pr^0(\omega^0)$, where $\omega^0$ is the set of equivalence classes of \eee contained in $\omega$. For all other elements $\omega$ of $\OO$, define $\Pr(\omega)=0$.  Now verifying~\eqref{req} is a routine matter.
\end{proof}

\section{Causal betweenness}\label{S:btw}

 Reichenbach \cite[p.~190]{Rei56} defined an event $B$ to be
 \emph{causally between} events $A$ and $C$ if
 \begin{align*}
    1 > \Pr(A|B) > \Pr(A|C)> \Pr(A) > 0\,,\\
    1 > \Pr(C|B) > \Pr(C|A)> \Pr(C) > 0\,,\\
    \Pr(C|AB) = \Pr(C|B)\,.\qquad\qquad
 \end{align*}
 Implicit in this definition is the assumption $\Pr(B)>0$
 that makes $\Pr(A|B)$ and $\Pr(C|B)$ meaningful. In turn,
 $\Pr(A|B)>0$ means $\Pr(AB)>0$, which makes $\Pr(C|AB)$ meaningful. If $B$
 is causally between $A$ and $C$, then all three events are
 $\Pr$-nontrivial and no two of them $\Pr$-equal.

 If $(A,B,C)$ is a conjunctive fork, then (contrary to the claim in~\cite[p.~179]{Bre77}) $B$ need not be causally
 between $A$ and $C$ even if no two of $A,B,C$ are $\Pr$-equal: for example, if
 \[
    \begin{array}{llll}
        \Pr(A BC)=1/5,   &\Pr(A B\oC)=1/5,\\ 
				\Pr(\oA BC)=1/5, &\Pr(\oA B\oC)=1/5,\\ 
				\Pr(A\oB C)=0,        &\Pr(A\oB\oC)=0,\\
        \Pr(\oA\oB C)=0, &\Pr(\oA\oB\oC)=1/5,
    \end{array}
 \]
then $(A,B,C)$ is a conjunctive fork and $\Pr(A|B) = \Pr(A|C)$. 

If an event $B$ is causally between $A$ and $C$, then $(A,B,C)$ need not
 be a conjunctive fork: for example, if
 \[
    \begin{array}{llll}
        \Pr(A BC)=1/20,   &\Pr(A B\oC)=2/20,\\ 
				\Pr(\oA BC)=2/20, &\Pr(\oA B\oC)=4/20,\\ 
				\Pr(A\oB C)=0,        &\Pr(A\oB\oC)=1/20,\\
        \Pr(\oA\oB C)=1/20, &\Pr(\oA\oB\oC)=9/20,
    \end{array}
 \]
then $B$ is causally between $A$ and $C$ and $\Pr(AC|\oB) \ne \Pr(A|\oB) \Pr(C|\oB)$.

Following~~\cite{ChvWu12}, we call a ternary relation $\bbb$ on a finite ground set $N$
an {\em abstract causal betweenness\/} if, and only if, 
 there are events $A_i$ with $i$ ranging over $N$ such that 
 \begin{equation*}\label{cbtw}
     \bbb=\{(i,j,k)\in N^3\colon\;\text{$A_j$ is causally between $A_i$ and $A_k$}\}\,.
 \end{equation*}
A natural question is which ternary~relations $\bbb$ form an abstract causal betweenness.  
 This question was answered in ~\cite[Theorem~1]{ChvWu12} in terms of
 the directed graph~$G(\bbb)$ whose vertices are all two-element
 subsets of $N$ and whose edges are all ordered pairs
 $(\{i,j\},\{i,k\})$ such that $(i,j,k)\in\bbb$ with $i,j,k$ pairwise
 distinct: 
\begin{align}
&\text{a ternary relation \bbb on a finite ground set} \label{ChvWu}\\
&\text{is an abstract causal betweenness if and only if}\nonumber \\ 
&\quad \bullet\; (i,j,k)\in\bbb\;\Rightarrow\; i,j,k \text{ are pairwise distinct,}\nonumber\\
&\quad \bullet\; (i,j,k)\in\bbb\;\Rightarrow\; (k,j,i)\in\bbb,\nonumber \\
&\quad \bullet\; \text{$G(\bbb)$ contains no directed~cycle.}\nonumber
\end{align}
(The third requirement implies that $(i,j,k)\in\bbb\;\Rightarrow\; (i,k,j)\not\in\bbb$: else $G(\bbb)$ would contain the directed cycle
$\{i,j\}\rightarrow\{i,k\}\rightarrow\{i,j\}$.)

An essential difference between abstract causal betweenness and fork-representable relations is that, on the one hand, every triple in an abstract causal betweenness consists of pairwise distinct elements and, on the other hand, a forkness includes with every triple $(i,j,k)$ most of triples formed by at most two of $i,j,k$. This difference notwithstanding, the two can be compared. The trick is to introduce, for every ternary relation \rrr, the ternary relation $\rrr^\sharp$ consisting of all triples  in \rrr that have pairwise distinct elements.

We claim that 
\begin{align}\label{comp}
&\text{if \rrr is a fork-representable relation on a finite ground set}\\
&\text{such that $(i,j,i)\not\in\rrr$ whenever $i\ne j$,}\nonumber\\
&\text{then $\rrr^\sharp$ is an abstract causal betweenness.}\nonumber
\end{align}
To justify this claim, consider a fork-representable relation \rrr on a finite set $N$ such that $(i,j,i)\not\in\rrr$ whenever $i\ne j$. By Lemma~\ref{V1}, \rrr is a forkness; assumption 
$i\ne j\Rightarrow (i,j,i)\not\in\rrr$ implies that \eee is the identity relation, and so the quotient of \rrr is isomorphic to \rrr. Now Lemma~\ref{V2} guarantees that \rrr is solvable: there are  positive numbers $x_{\{i,j\}}$ such that 
$x_{\{i,k\}}=x_{\{i,j\}}+x_{\{j,k\}}$ for all $(i,j,k)$ in $\rrr$ with pairwise distinct $i,j,k$. Since 
$x_{\{i,k\}}>x_{\{i,j\}}$ for every edge $(\{i,j\},\{i,k\})$ of~$G(\rrr^\sharp)$, this directed graph is acyclic, and so~\eqref{ChvWu} guarantees that $\rrr^\sharp$ is an abstract causal betweenness.

Assumption $i\ne j\Rightarrow (i,j,i)\not\in\rrr$ cannot be dropped from~\eqref{comp}: consider $\rrr = N^3$. This \rrr is fork representable (for instance, by $\OO=\{x,y\}$, $\Pr(x)=\Pr(y)=1/2$, and $A_i=\{x\}$ for all $i$ in $N$). Nevertheless, if $\abs{N}\ge 3$, then $G(\rrr^\sharp)$ contains cycles, and so $\rrr^\sharp$ is not an abstract causal betweenness.

The converse of~\eqref{comp}, 
\begin{align*}
&\text{if $\rrr^\sharp$ is an abstract causal betweenness}\phantom{xxxxxxxxxxxxxxxxxx}\\
&\text{then \rrr is a fork-representable relation}\\
&\text{such that $(i,j,i)\not\in\rrr$ whenever $i\ne j$,}
\end{align*}
is false. Even its weaker version,
\begin{align*}
&\text{if \rrr is a regular forkness}\phantom{xxxxxxxxxxxxxxxxxxxxxxxxxxxxi}\\
&\text{such that $\rrr^\sharp$ is an abstract causal betweenness,}\\
&\text{then \rrr is a fork-representable relation,}
\end{align*}
is false: consider the smallest forkness \rrr on $\{1,2,3,4\}$
that contains the relation
\begin{align*}
\{&(1,3,2), (2,3,4), (3,1,4), (1,4,2), \label{example}\\
&(2,3,1), (4,3,2), (4,1,3), (2,4,1)\}. \nonumber
\end{align*}
Minimality of \rrr implies that 
$(i,j,i)\not\in\rrr$ whenever $i\ne j$; it follows that \eee is the identity relation, and so \rrr 
is a regular forkness. Graph $G(\rrr^\sharp)$ is acyclic
  \begin{center}\setlength{\unitlength}{1.3mm}\setlength{\fboxsep}{0.5mm}
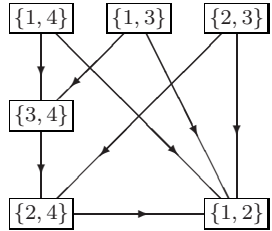
    \begin{picture}(40,29)(0,7)
     \put(10,30){\vector(0,-1){6}}\put(10,30){\line(0,-1){10}}
     \put(20,30){\vector(-1,-1){7}} \put(20,30){\line(-1,-1){10}}
     \put(10,30){\vector(1,-1){14}}\put(10,30){\line(1,-1){20}}
     \put(30,30){\vector(-1,-1){14}} \put(30,30){\line(-1,-1){20}}
     \put(10,20){\vector(0,-1){6}}\put(10,20){\line(0,-1){10}}
     \put(20,30){\vector(1,-2){6}} \put(20,30){\line(1,-2){10}}
     \put(10,10){\vector(1,0){11}}\put(10,10){\line(1,-0){20}}
     \put(30,30){\vector(0,-1){12}} \put(30,30){\line(0,-1){20}}
     \put(10,20){\makebox(0,0){\fcolorbox{black}{white}{\tiny$\{3,4\}$}}}
     \put(10,10){\makebox(0,0){\fcolorbox{black}{white}{\tiny$\{2,4\}$}}}
     \put(30,10){\makebox(0,0){\fcolorbox{black}{white}{\tiny$\{1,2\}$}}}
     \put(10,30){\makebox(0,0){\fcolorbox{black}{white}{\tiny$\{1,4\}$}}}
     \put(20,30){\makebox(0,0){\fcolorbox{black}{white}{\tiny$\{1,3\}$}}}
     \put(30,30){\makebox(0,0){\fcolorbox{black}{white}{\tiny$\{2,3\}$}}}
   \end{picture}
 \end{center}
and so \eqref{ChvWu} guarantees that $\rrr^\sharp$ is an abstract causal betweenness.
By Lemma~\ref{V2}, \rrr is not fork representable: here, system~\eqref{solution} is isomorphic to
 \begin{eqnarray*}
    x_{\{1,2\}}=x_{\{1,3\}}+x_{\{2,3\}} \\
    x_{\{2,4\}}=x_{\{2,3\}}+x_{\{3,4\}} \\
    x_{\{3,4\}}=x_{\{1,3\}}+x_{\{1,4\}} \\
    x_{\{1,2\}}=x_{\{1,4\}}+x_{\{2,4\}}
 \end{eqnarray*}
and this system has no solution with $x_{\{1,4\}}>0$ as the linear combination of its four equations
with multipliers $-1$, $+1$, $+1$, $+1$ reads $0=2x_{\{1,4\}}$.

\section{Concluding remarks}\label{S:disc} 

1.  The patterns studied in this work are based on combinations of conditional
 independence and covariance constraints for events. In recent  decades,
 patterns of conditional independence among random variables have been studied
 in statistics and in probability theory since they provide insight to decompositions
 of multidimensional distributions, so sought for in applications. A framework
 for this activity was developed in the graphical
 models community\cite{Lauri}.

 A general formulation of the problem considers random variables
 $\xi_i$ indexed by $i$ in $N$ and patterns consisting of the conditional
 independences $\xi_i\ci\xi_j|\xi_K$ where $\xi_K=(\xi_k)_{k\in K}$,
 $i,j\in N$, and $i,j\not\in K$. The case $i=j$ means functional
 dependence of $\xi_i$ on $\xi_K$, a.s. The problem is highly nontrivial
 even for four variables~\cite{M.4var.III}.

 First treatments go back to~\cite{Pearl,Spo80}. The variant of the problem
 excluding the functional dependence is most frequent \cite{Studeny}.
 Restrictions to Gaussian~\cite{M.Lnenicka,sul} or binary variables,
 positivity of the distribution of $\xi_N$, etc., have been studied
 as well \cite{Dawid}. The idea to employ the Fourier-Stieltjes transform,
 as in Section~\ref{S:mainproof}, appeared in~\cite{M.indep}, characterizing
 patterns of unconditional independence.

 2. For patterns of conditional independence, the role of forkness is played 
 by graphoids \cite{Pearl}, semigraphoids \cite{semigr},
 imsets \cite{Studeny}, semimatroids \cite{M.4var.III}, etc. Notable are
 connections to matroid representations theory, see \cite{M.matroid}.

3.  All possible patterns of conjunctive forks on events $A_i$ indexed by $i$ in $N$ 
  arise by varying a probability measure on \OO , the power set of $N$. For a ternary relation \rrr on 
	a finite set $N$, the set
 $\mathcal P_\rrr$ of probability measures $\Pr$ on \OO that satisfy $(i,j,k)\in \rrr \;\Leftrightarrow\;
 (A_i,A_j,A_k)_{\Pr}$  is described by finitely many constraints that require quadratic
 polynomials in in\-deter\-minates $z_{\omega}$ indexed by $\omega$ in $\OO$
 to be positive or zero. For fork-representability
 of~$\rrr$, it matters only whether $\mathcal P_\rrr$
 is empty or not, which can be found out in polynomial time by the main result of the present paper.
 The shape of~$\mathcal P_\rrr$, which is a semialgebraic subset of the probability simplex, 
 might be difficult to understand; to reveal it, finer algebraic
 techniques are needed, as in algebraic statistics \cite{DoSS09,Zwi}.

4. One of the two {\em Discrete Mathematics\/} reviewers asked: ``Is there some interesting algebraic/combinatorial structure
in admissible forknesses? Can they be partially ordered for a fixed $N$?" We leave these questions open.

\section*{Acknowledgments}

This work began in March 2011 in the weekly meetings of the seminar
ConCoCO (Concordia Computational Combinatorial Optimization). We wish
to thank its participants, in particular, Luc Devroye and Peter Sloan,
for stimulating and helpful interactions. We also thank the two {\em Discrete Mathematics\/}   
reviewers for their thoughtful comments that 
made us improve the presentation considerably.

\end{document}